\documentclass[preprint,1p,fleqn,11pt,letterpaper]{elsmodified}
\usepackage{multicol}
\usepackage{fancyvrb}
\usepackage{amsmath,amssymb}

\pdfoutput=1
\usepackage{graphicx}
\usepackage{geometry}
\usepackage[usenames]{color}
\usepackage[colorlinks,citecolor=blue,urlcolor=blue,pagebackref=true,linkcolor=black]{hyperref}
\usepackage{hypernat}

\newcommand{\tuple}[1]{\langle #1 \rangle}

\newcommand{\bA}{A}
\newcommand{\boldA}{\mathbf{A}}

\newcommand{\Intervals}{\mathcal{I}}
\newcommand{\RealIntervals}{\Intervals_\Reals}
\newcommand{\RationalIntervals}{\Intervals_\Rationals}
\newcommand{\BorelSets}{\mathcal{B}_{\Reals}}

\newcommand{\Alg}{\mathcal{L}}
\newcommand{\RAlg}{\Alg_{\Reals}}
\newcommand{\QAlg}{\Alg_{\Rationals}}
\newcommand{\RAlgk}[1]{\Alg_{\Reals^#1}}
\newcommand{\QAlgk}[1]{\Alg_{\Rationals^#1}}

\newcommand{\ro}{{<}} 
\newcommand{\dom}{\textrm{dom}}

\newcommand{\Realbold}[1]{{#1}} 
\newcommand{\RealV}{{\Realbold{V}}}
\newcommand{\vC}{C}
\newcommand{\Rpsi}{{\boldsymbol{\psi}}}
\newcommand{\Rzeta}{{\boldsymbol{\zeta}}}
\newcommand{\RB}{{\boldsymbol{\beta}}}
\newcommand{\RG}{{\boldsymbol{\gamma}}}
\newcommand{\RZ}{{\boldsymbol{\varphi}}}

\newcommand{\Measures}{\mathcal{M}}
\newcommand{\ProbMeasures}{\Measures_1}
\newcommand{\cS}{\mathcal{S}}
\newcommand{\cT}{\mathcal{T}}
\newcommand{\cR}{\mathcal{R}}
\newcommand{\bbI}{\mathbb{I}}

\newcommand{\Expect}{\mathbb{E}}

\newcommand{\Beta}{\textrm{Beta}}

\renewcommand{\Pr}{\mathbb{P}}
\newcommand{\defas}{:=}
\newcommand{\given}{\mid}
\newcommand{\st}{\,:\,}
\newcommand{\Ind}{\mathbf 1}
\newcommand{\closure}{\overline}

\newcommand{\pars}{\,\cdot\,}

\def\[#1\]{\begin{align}#1\end{align}} 

\newcommand{\ohm}{{\mathrm \Omega}}

\newcommand{\arrowlt}{\vartriangleleft}

\newcommand{\arrowgt}{\vartriangleright}

\newcommand{\Reals}{\ensuremath{\mathbf{R}}}
\newcommand{\Rationals}{\ensuremath{\mathbf{Q}}}

\newcommand{\flip}{{\sl\ttfamily{flip}}}

\newtheorem{theorem}{Theorem}[section]
\newtheorem{lemma}[theorem]{Lemma}
\newtheorem{corollary}[theorem]{Corollary}
\newtheorem{proposition}[theorem]{Proposition}
\newdefinition{definition}[theorem]{Definition}
\newproof{proof}{Proof}
\newproof{cdfproof}{Proof of Theorem~\ref{CdeFtheorem}}

\journal{Annals of Pure and Applied Logic}

\begin{document}

\begin{frontmatter}

\title{Computable de Finetti measures}

\author[cef]{Cameron E.\ Freer}

\author[dmr]{Daniel M.\ Roy}

\address[cef]{Department of Mathematics,
Massachusetts Institute of Technology, Cambridge, MA, USA}

\address[dmr] {Computer Science and Artificial Intelligence Laboratory,\\
Massachusetts Institute of Technology, Cambridge, MA, USA}

\begin{abstract}
We prove a computable version of de~Finetti's  theorem on
exchangeable sequences of real random variables.  As a consequence,
exchangeable stochastic processes expressed in probabilistic functional
programming languages can be automatically rewritten as procedures
that do not modify non-local state.  Along the way, we prove that a
distribution on the unit interval is computable if and only if its
moments are uniformly computable.
\end{abstract}


\end{frontmatter}


\section{Introduction}
The classical  de~Finetti theorem states that an exchangeable
sequence of real random variables is a mixture of independent and
identically distributed (i.i.d.)\ sequences of random variables.
Moreover, there is an (almost surely unique) measure-valued random
variable, called the \emph{directing random measure}, conditioned on
which the random sequence is i.i.d.  The distribution of the
directing random measure is called the \emph{de~Finetti measure} or
the \emph{mixing measure}.

This paper examines the \emph{computable} probability theory of
exchangeable sequences of real-valued random variables.  We prove a
computable version of de~Finetti's theorem: 
the distribution of an exchangeable sequence of real random
variables is computable if and only if its de~Finetti
measure is computable.  
The classical proofs do not readily effectivize; instead,
we show how to 
directly compute the de~Finetti measure (as characterized by the
classical theorem) in terms of a computable representation of the
distribution of the exchangeable sequence.  Along the way, we prove
that a distribution on $[0,1]^\omega$ is computable if and only if
its moments are uniformly computable, which may be of independent
interest.

A key step in the proof is to describe the de~Finetti measure in
terms of the moments of a set of random variables derived from the
exchangeable sequence.  When the directing random measure is (almost
surely) continuous, we can show that these moments are computable,
which suffices to complete the proof of the main theorem in this
case.  In the general case, we give a proof inspired by a randomized
algorithm that, 
with probability one, 
computes the de~Finetti measure.

\subsection{Computable Probability Theory}

These results are formulated in the Turing-machine-based bit-model
for computation over the reals (for a general survey, see Braverman
and Cook \citep{MR2208383}).  This computational model has been
explored both via the type-2 theory of effectivity (TTE) framework
for computable analysis, and via effective domain-theoretic
representations of measures. 

Computable analysis has its origins in the study of recursive
real functions, and can be seen as a way to provide ``automated
numerical analysis'' (for a tutorial, see Brattka, Hertling, and
Weihrauch~\citep{brattkatutorial}).  Effective domain theory has its
origins in the study of the semantics of programming languages, where it
continues to have many applications (for a survey, see
Edalat~\citep{MR1619397}).  Here we use methods from these
approaches to transfer a representational result from probability
theory to a setting where it can directly transform statistical
objects as represented on a computer.

The computable probability measures in the bit-model coincide with
those distributions from which we can generate exact samples to
arbitrary precision on a computer.  Results in the bit-model also have direct
implications for programs that manipulate probability distributions numerically.
In many areas of
statistics and computer science, especially machine learning, the
objects of interest include
distributions on data structures that are
higher-order or are defined using recursion.  Probabilistic
functional programming languages provide a convenient setting for
describing and manipulating such distributions, and the theory we present here is directly relevant to this setting.

Exchangeable sequences play a fundamental role in both
statistical models and their implementation on computers.
Given a \emph{sequential} 
description of an exchangeable process, in
which one uses previous samples or sufficient statistics to sample
the next element in the sequence, a direct implementation in a
probabilistic functional programming
language would need to use non-local communication (to access old
samples or update sufficient statistics).  This is often implemented
by modifying the program's internal state directly (i.e., using
\emph{mutation}), or via some indirect method such as a state monad.
The classical de~Finetti theorem implies that (for such sequences
over the reals) there is an alternative description in which samples
are conditionally independent (and so could be implemented without
non-local communication), thereby allowing parallel implementations.
But the classical result does not imply that there is a
\emph{program} that samples the sequence according to this
description. 
Even when there is such a program, the classical theorem does not
provide a method for finding it.
The computable de~Finetti theorem states that such a program
\emph{does} exist. Moreover, the proof itself provides a
\emph{method} for constructing the desired program.
In Section~\ref{funcpure} we describe how an implementation of the
computable de~Finetti theorem would perform a code transformation that
eliminates the use of non-local state in procedures that induce
exchangeable stochastic processes.

This transformation is of interest beyond its implications for
programming language semantics.  In statistics and machine learning,
it is often desirable to know the representation of an exchangeable
stochastic process in terms of its de~Finetti measure (for several
examples, see Section~\ref{partialexch}).  Many such processes in
machine learning have very complicated (though computable)
distributions, and it is not always feasible to find the de~Finetti
representation by hand.  The computable de~Finetti theorem provides
a method for automatically obtaining such representations.

\section{de~Finetti's Theorem}

We assume familiarity with the standard measure-theoretic
formulation of probability theory (see, e.g., Billingsley
\citep{MR1324786} or Kallenberg \citep{MR1876169}).
Fix a basic probability space $(\ohm, \mathcal{F}, \Pr)$ and let
$\BorelSets$ denote the Borel sets of $\Reals$.  Note that we will
use $\omega$ to denote the set of nonnegative integers (as in
logic), rather than an element of the basic probability space $\ohm$
(as in probability theory).  By a \emph{random measure} we mean a
random element in the space of Borel measures on $\Reals$, i.e., a
kernel from $(\ohm, \mathcal{F})$ to  $(\Reals,\BorelSets)$.  An
event $A \in \mathcal{F}$ is said to occur \emph{almost surely}
(a.s.)\ if $\Pr A = 1$.  We denote the indicator function of a set
$B$ by $\Ind_B$.

\begin{definition}[Exchangeable sequence]
Let $X=\{X_i\}_{i\ge1}$ be a sequence of real-valued random variables. 
We say that $X$ is \emph{exchangeable} if, for every finite set
$\{k_1,\dotsc,k_j\}$ of distinct indices,
$(X_{k_1}, \ldots, X_{k_j})$ is equal in distribution to $(X_1, \ldots, X_j)$.
\end{definition}

\begin{theorem}[de~Finetti {\citep[][Chap. 1.1]{MR2161313}}]
Let
$X=\{X_i\}_{i\ge1}$
be an exchangeable sequence of real-valued random variables.  
There is a random probability measure $\nu$ on $\Reals$ such that 
$\{X_i\}_{i\ge 1}$ is conditionally i.i.d.\ with respect to $\nu$. That is,
\[
\Pr [X \in \pars \given \nu \,] = \nu^\infty \quad \mathrm{a.s.} \label{ciid}
\]
Moreover, $\nu$ is a.s.\ unique and given by
\[
\label{aslimit}
\nu(B) = \lim_{n\to\infty} \frac1n {\sum_{i=1}^n \Ind_B(X_i)}\quad \mathrm{a.s.}, 
\]
where $B$ ranges over $\BorelSets$.
\qed\end{theorem}
The random measure $\nu$ is called the \emph{directing random
measure}.\footnote{
The directing random
measure is only unique up to a null set, but it is customary to 
refer to it as if it were unique,
as long as we only rely on almost-sure properties.}  
Its distribution (a measure on probability measures),
which we denote by $\mu$, is called the \emph{de~Finetti measure} or
the \emph{mixing measure}.
As in Kallenberg \citep[][Chap. 1, Eq. 3]{MR2161313}, we may take
expectations on both sides of \eqref{ciid} to arrive at a
characterization 
\[
\Pr\{X \in \pars \} = \Expect \nu^\infty = \int m^\infty \mu(dm)
\label{mixture}
\]
of an exchangeable sequence as a mixture of i.i.d.\ sequences.

A Bayesian perspective suggests the following interpretation:
exchangeable sequences arise from independent observations from a
latent measure $\nu$.  
Posterior analysis follows from placing a prior distribution on
$\nu$.  
For further discussion of the implications of de~Finetti's theorem
for the foundations of statistical inference, see
Dawid~\citep{MR675977} and Lauritzen~\citep{MR754971}. 

In 1931, de~Finetti \citep{deFinetti} proved the classical result
for binary exchangeable sequences, in which case the de~Finetti
measure is simply a mixture of Bernoulli distributions; the
exchangeable sequence is equivalent to repeatedly flipping a coin
whose weight is drawn from some distribution on $[0,1]$.  
In 1937, de~Finetti \citep{MR1508036} extended the result to arbitrary 
real-valued exchangeable sequences.  We will refer to this more
general version as the \emph{de~Finetti theorem}.  
Later, Hewitt and Savage~\citep{MR0076206} 
extended the result to compact Hausdorff spaces,
and Ryll-Nardzewski~\citep{MR0088823} introduced a weaker notion
than exchangeability that suffices to give a conditionally i.i.d.\
representation.
Hewitt and Savage~\citep{MR0076206} provide a history of the early
developments, and a discussion of some subsequent extensions can be
found in Kingman~\citep{MR0494344}, Diaconis and
Freedman~\citep{MR786142}, and Aldous~\citep{MR883646}.  A recent
book by Kallenberg~\citep{MR2161313} provides a comprehensive view
of the area of probability theory that has grown out of de~Finetti's
theorem, stressing the role of invariance under symmetries.

\subsection{Examples}
\label{polyaex}
Consider an exchangeable sequence of $[0,1]$-valued random
variables.  In this case, the de~Finetti measure is a distribution
on the (Borel) measures on $[0,1]$.  For example, if the de~Finetti
measure is a Dirac measure on the uniform distribution on $[0,1]$
(i.e., the distribution of a random measure which is almost surely
the uniform distribution), then the induced exchangeable sequence
consists of independent, uniformly distributed random variables on
$[0,1]$.

As another example, let $p$ be a random variable, uniformly
distributed on $[0,1]$, and let $\nu\defas\delta_p$, i.e., the Dirac
measure concentrated on $p$.  Then the de~Finetti
measure is the uniform distribution on Dirac measures on $[0,1]$,
and the corresponding exchangeable sequence is $p,p,\dotsc$, i.e., a
constant sequence, marginally uniformly distributed.

As a further example, we consider a stochastic process
$\{X_i\}_{i\ge 1}$ 
composed of binary random variables
 whose finite marginals are given by
\[
\Pr \{X_1 = x_1,\, \ldots,\, X_n = x_n\} =
\frac{\Gamma(\alpha+\beta)}{\Gamma(\alpha)\Gamma(\beta)}\,
\frac{\Gamma(\alpha+ S_n)\Gamma(\beta+(n - S_n))}{\Gamma(\alpha+\beta+n)},
\label{marginals}
\]
where $S_n \defas \sum_{i\le n} x_i$, and where
$\Gamma$ is the Gamma function and $\alpha, \beta$ are
positive real numbers.  (One can verify that these
marginals satisfy Kolmogorov's extension theorem
\citep[][Theorem~6.16]{MR1876169}, and so there is a
stochastic process $\{X_i\}_{i\ge 1}$  with these finite marginals.)
Clearly this process is exchangeable, as $n$ and $S_n$
are invariant to order.
This process can also be described by a sequential scheme known as
P\'olya's urn \citep[][Chap.~11.4]{MR1093667}.
Each $X_i$ is sampled in turn according to the conditional
distribution
\[
\Pr\{X_{n+1}= 1 \given X_1 = x_1,\, \ldots,\, X_n = x_n\} =
\frac{\alpha + S_n}{\alpha + \beta + n}.
\]
This process is often described as repeated sampling from an urn: starting with $\alpha$ red
balls and $\beta$ black balls, a ball is drawn at each stage
uniformly at random, and then returned to the urn along with an
additional ball of the same color.
By de~Finetti's theorem, there exists a random variable $\theta \in [0,1]$
with respect to which the sequence is conditionally independent and
$\Pr \{ X_i = 1 \given \theta\} = \theta$ for each $i$.
In fact, 
\[\textstyle
\Pr[X_1 = x_1,\, \ldots,\, X_n = x_n \given \theta\,] =
\prod_{i\le n} \Pr [ X_i = x_i \given \theta\, ]
= \theta^{S_n} (1- \theta)^{(n- S_n)}.
\]
Furthermore, one can show that $\theta$ is
$\Beta(\alpha, \beta)$-distributed,
and so the process given by the marginals \eqref{marginals} is
called the Beta-Bernoulli process.  
Finally, the de~Finetti measure is
the distribution of the random Bernoulli measure $\theta \delta_1 +
(1-\theta) \delta_0$.

\subsection{The Computable de~Finetti Theorem}

In each of these examples, the de~Finetti measure is a 
\emph{computable measure}. (In Section~\ref{representation}, we
make this and related notions precise. For an implementation of the Beta-Bernoulli process in a probabilistic programming language, see 
Section~\ref{funcpure}.)
A natural question to ask is whether computable exchangeable
sequences always arise from computable de~Finetti measures.
In fact, computable de~Finetti
measures give rise to computable distributions on exchangeable sequences (see Proposition~\ref{mucomp}).
Our main result is the converse:
every computable distribution on real-valued exchangeable sequences 
arises from a computable de~Finetti measure.

\begin{theorem}[Computable de~Finetti] \label{CdeFtheorem} 
Let $\chi$ be the distribution of a real-valued exchangeable
sequence $X$, and let $\mu$ be the distribution of its directing
random measure $\nu$.  Then $\mu$ is 
computable relative to $\chi$, and $\chi$ is computable relative to
$\mu$.  In particular, $\chi$
is computable if and only if $\mu$ is computable.
\end{theorem}

The directing random measure is classically given 
a.s.\ by the explicit limiting expression \eqref{aslimit}.
Without a computable handle on the rate of convergence, the limit is not
directly computable, and so we cannot use this limit directly to
compute the de~Finetti measure.
However, we are able to
reconstruct the de~Finetti measure using the moments of 
random variables derived from the directing random measure.

\subsubsection{Outline of the Proof}\label{ideas}

Recall that $\BorelSets$ denotes the Borel sets of $\Reals$.  Let
$\RealIntervals$ denote the set of open intervals, and let
$\RationalIntervals$ denote the set of open intervals with rational
endpoints.  Then $\RationalIntervals \subsetneq \RealIntervals
\subsetneq \BorelSets$.   For $k \ge 1$ and
$\RB \in \BorelSets^k = \BorelSets \times \dotsm \times \BorelSets$, 
we write $\RB(i)$ to denote the $i$th coordinate of $\RB$.

Let $X=\{X_i\}_{i\geq 1}$ be an exchangeable sequence of real random
variables, with distribution $\chi$ and directing random measure
$\nu$.  For every $\RG \in \BorelSets$, we define a $[0,1]$-valued
random variable $V_{\RG} \defas \nu \RG$. 
A classical result in probability theory 
\citep[][Lem.~1.17]{MR1876169} implies that a Borel
measure on $\Reals$ is uniquely characterized by the mass it places
on the open intervals with rational endpoints.
Therefore, the distribution of the stochastic process $\{V_\tau\}_{\tau \in \RationalIntervals}$ determines the de~Finetti measure $\mu$ (the distribution of $\nu$).
\begin{definition}[Mixed moments]
Let $\{x_i\}_{i\in C}$ be a family of random variables indexed by a set $C$.  The \emph{mixed moments of $\{x_i\}_{i\in C}$} are the expectations 
$\Expect \bigl(\prod_{i=1}^k x_{j(i)}\bigr)$, for $k\ge1$ and $j \in C^k$.
\end{definition}
We can now restate the consequence 
of de~Finetti's theorem described in Eq.~\eqref{mixture},
in terms of the 
finite-dimensional marginals of the exchangeable sequence $X$ and
the mixed moments of 
$\{V_\RB\}_{\RB \in \BorelSets}$.

\begin{corollary} \label{link1}
$
\Pr \bigl ( \bigcap_{i=1}^k  \{X_i \in \RB(i)  \} \bigr ) 
= \Expect \bigl (\prod_{i=1}^k  V_{\RB(i)}  \bigr)
$ 
for $k\ge 1$ and $\RB \in \BorelSets^k$.
\qed\end{corollary}
For $k\ge 1$, let $\RAlgk{k}$ denote the set of finite unions of open rectangles in
$\Reals^k$
(i.e., the lattice generated by $\RealIntervals^k$),
and let $\QAlgk{k}$ denote the 
set of finite unions of open rectangles in $\Rationals^k$.
(Note that $\RationalIntervals \subsetneq \QAlg \subsetneq \RAlg
\subsetneq \BorelSets$.)
As we will show in Lemma~\ref{compdiststandard}, when $\chi$ is
computable, 
we can enumerate all rational lower bounds
on quantities of the form
\[\textstyle
\Pr \bigl ( \bigcap_{i=1}^k  \{X_i \in \sigma(i)  \} \bigr ),
\label{ratbound}
\]
where $k\ge 1$ and $\sigma\in \QAlg^k$.

In general, we cannot enumerate all rational upper bounds on
\eqref{ratbound}.  However, if $\sigma\in\QAlg^k$ (for $k\ge1$) is such that, with probability one,
$\nu$ places no mass on the
boundary of
any $\sigma(i)$, then
$\Pr \bigl ( \bigcap_{i=1}^k  \{X_i \in \sigma(i)  \} \bigr )
= \Pr \bigl ( \bigcap_{i=1}^k  \{X_i \in {\closure{\sigma(i)}}  \}
\bigr )$,
where $\closure {\sigma(i)}$ denotes the closure of $\sigma(i)$.
In this case, for every rational upper bound $q$ on \eqref{ratbound},
we have that $1-q$ 
is a lower bound on 
\[\textstyle
\Pr \bigl ( \bigcup_{i=1}^k  \{X_i \not \in \closure{\sigma(i)}  \} \bigr ),
\]
a quantity for which we can enumerate all rational lower bounds.
If this property holds for all $\sigma \in \QAlg^k$,
then we can compute the mixed moments
$\{V_\tau\}_{\tau \in \QAlg}$.
A natural condition that implies this property for all
$\sigma\in\QAlg^k$ is
that $\nu$ is a.s.\ continuous (i.e.,
with probability one, $\nu\{x\} = 0$ for every $x \in \Reals$).

In Section~\ref{momentproblem}, we show how to computably recover a
distribution from its moments.  This suffices to recover the
de~Finetti measure when $\nu$ is a.s.\ continuous, as we show in
Section~\ref{easycase}.
In the general case, point masses in $\nu$ can prevent us from
computing the mixed moments.  Here we use a proof inspired by a
randomized algorithm that almost surely avoids the point masses and
recovers the de~Finetti measure.  For the complete proof, see
Section~\ref{mainproof}.


\section{Computable Representations}
\label{representation}

We begin by
introducing notions of computability on various spaces.
These definitions follow from more general TTE notions, though we
will sometimes derive simpler equivalent representations for the
concrete spaces we need (such as the real numbers, Borel measures on
reals, and Borel measures on Borel measures on reals).  For details,
see the original papers, as noted.

We assume familiarity with standard notions of computability
theory, such as computable and computably enumerable (c.e.)\ sets (see, e.g., Rogers
\citep{MR886890} or \citet{MR882921}).  Recall that  $r \in \Reals$
is a \emph{c.e.\ real} (sometimes called a \emph{left-c.e.}\ or
\emph{left-computable real})
when the set of all rationals less than $r$ is a c.e.\ set.
Similarly,  $r$ is a \emph{co-c.e.\ real} (sometimes called a
\emph{right-c.e.}\ or \emph{right-computable real}) when the set of all rationals greater than
$r$ is c.e. A real $r$ is a computable real when it is both a 
c.e.\ and co-c.e.\ real.

To represent more general spaces, we work in terms of an effectively
presented topology.
Suppose that $S$ is a second-countable $T_0$ topological space with
subbasis $\cS$.
For every point $x\in S$, define the set
$\cS_x \defas \{B \in \cS \st x \in B\}$.  
Because $S$ is $T_0$, we have $\cS_x \neq \cS_y$ when
$x\neq y$, and so the set $\cS_x$ uniquely determines the point $x$.
It is therefore convenient to define representations on
topological spaces 
under the assumption that the space is $T_0$.
In the specific cases below, we often have much more structure,
which we use to simplify the representations.

We now develop these definitions more formally.

\begin{definition}[Computable topological space]\label{ctzero}
Let $S$ be a second-count\-able $T_0$ topological space with a
countable subbasis $\cS$.
Let $s : \omega \to \cS$ be an enumeration of $\cS$ (possibly with
repetition), i.e., a total surjective (but not necessarily injective) function.
We say that $S$ is a \emph{computable topological space (with respect to $s$)} when the set
\[
\label{explicit}
\bigl\{ \langle m, n \rangle \st s(m) = s(n) \bigr\}
\]
is a c.e.\ subset of $\omega$, where $\langle \pars,\pars \rangle$ is a standard
pairing function.
\end{definition}

This definition of a computable topological space is derived
from Weihrauch's definition \citep[][Def.~3.2.1]{358357} in terms
of ``notations''. 
(See also, e.g., Grubba, Schr\"oder, and Weihrauch
\citep[][Def.~3.1]{MR2351939}.)

It is often possible to pick a subbasis $\cS$ (and enumeration $s$)
for which the elemental
``observations'' that one can computably observe are those of the form
$x\in B$, where $B \in \cS$.  Then the set 
$\cS_x = \{B \in \cS \st x \in B\}$
is computably enumerable (with respect to $s$) when the point $x$ is
such that it is eventually noticed to be in each basic open set
containing it; we will call such a point $x$ \emph{computable}. 
This is one motivation
for the definition of computable point in a $T_0$ space below.

Note that in a $T_1$ space, two computable points are computably
distinguishable, but in a $T_0$ space, computable points will be, in
general, distinguishable only in a computably enumerable fashion.
However, this is essentially the best that is possible, if the open
sets are those that
we can ``observe''.
(For more details on this approach 
to considering datatypes as topological spaces, in which basic open
sets correspond to ``observations'', see
Battenfeld, Schr\"oder, and Simpson \citep[][\S2]{MR2328287}.)
Note that the choice of topology and subbasis are essential; for
example, we can recover both computable reals and c.e.\ reals as
instances of ``computable point'' for appropriate computable
topological spaces, as
we describe in Section~\ref{repreals}.

\begin{definition}[Names and computable points]
Let $(S,\cS)$ be a comput\-able topological space with respect to an enumeration $s$.
Let $x \in S$.  The set
\[
\label{comppoints}
\{ n : s(n) \in \cS_x \} = \{ n : x \in s(n) \}
\]
is called the \emph{$s$-name} (or simply,
\emph{name}) of $x$.  
We say that $x$ is computable when its $s$-name is c.e.
\end{definition}
Note that this use of the term ``name'' 
is similar to the notion of a ``complete name'' 
(see \citep[][Lem.~3.2.3]{358357}),
but differs somewhat from TTE usage
(see \citep[][Def.~3.2.2]{358357}).

\begin{definition}[Computable functions]\label{cfuncs1}
Let $(S,\cS)$ and $(T,\cT)$ be comput\-able topological spaces (with respect to 
enumerations $s$ and $t$, respectively).
We say that a function 
$f: S \to T$ is
\emph{computable (with respect to $s$ and $t$)} when there is 
a partial computable functional
$g: \omega^\omega \to \omega^\omega$ such that for all $x \in
\dom(f)$ and enumerations $N=\{n_i\}_{i\in\omega}$ of an $s$-name of $x$, we
have that $g(N)$ is an enumeration of a $t$-name of $f(x)$.
\end{definition}

(See \citep[][Def.~3.1.3]{358357} for more details.)  Note that an implication of this definition is that computable functions are continuous.

Recall that a functional $g: \omega^\omega \to \omega^\omega$ is
partial computable if there is a monotone computable function $h :
\omega^{<\omega} \to \omega^{<\omega}$ mapping finite prefixes (of
integer sequences) to finite prefixes, such that given increasing
prefixes of an input $N$ in the domain of $g$, the output of $h$
will eventually include every finite prefix of $g(N)$.
(See \citep[][Def.~2.1.11]{358357} for more details.)
Informally, $h$ can be used to read in an enumeration of an $s$-name of a
point $x$ and outputs an enumeration of a $t$-name of
the point $f(x)$.

Let $(S,\cS)$ and $(T,\cT)$ be computable topological spaces.
In many situations where we are interested in establishing the
computability of some function $f: S \to T$, we may refer to the
function implicitly via pairs of points $x \in S$ and $y \in T$
related by $y = f(x)$.  In this case, we will say that $y$
(under the topology $\cT$)
is \emph{computable relative to} $x$ (under the topology $\cS$)
when $f: S\to T$  is a computable function.
We will often elide one or both topologies when they are clear from context.

\subsection{Representations of Reals}
\label{repreals}

We will use both the standard topology and right order topology on the real line $\Reals$.
The reals under the standard topology are a computable topological
space using the basis $\RationalIntervals$  with respect to a
straightforward effective enumeration; the computable points of this
space are the computable reals. 
The reals under the \emph{right order topology} are a computable
topological space using the basis 
\[
\cR_\ro \defas \bigl\{(c,\infty) \st c\in \Rationals \bigr\},
\]
under a standard enumeration; the computable points of this space
are the c.e.\ reals.  

Recall that, for $k\ge 1$, the set $\RationalIntervals^k$ is a basis for the (product of the) standard topology on $\Reals^k$
that is closed under intersection
and makes $(\Reals^k,\RationalIntervals^k)$ a computable topological space (under a straightforward enumeration of $\RationalIntervals^k$).  
Likewise, an effective enumeration of cylinders $\sigma \times \Reals^\omega$, for $\sigma \in \bigcup_{k\ge1} \RationalIntervals^k$, makes $\Reals^\omega$ a computable topological space.  Replacing $\RationalIntervals$ with $\cR_\ro$ and ``standard'' with ``right order'' above gives a characterization of computable vectors and sequences of reals under the right order topology.

We can use the right order topology to define a representation for open sets.
Let $(S,\cS)$ be a computable topological space, with respect to an enumeration $s$.  
Then an open set $B \subseteq S$ is \emph{c.e.\ open} when the indicator function $\Ind_B$ is computable with respect to $\cS$ and $\cR_\ro$.  
The c.e.\ open sets can be shown to be the computable points in the space of open sets under the Scott topology.
Note that for the computable topological space $\omega$ (under the discrete
topology and the identity enumeration)
the c.e.\ open sets are precisely the c.e.\ sets of naturals.

\subsection{Representations of Continuous Real Functions}

We now consider computable representations for continuous functions
on the reals.

Let $(S,\cS)$ and $(T,\cT)$ each be either of $(\Reals,\RationalIntervals)$ or
$(\Reals,\cR_\ro)$, and let $s$ and $t$ be the associated enumerations. 
For $k\ge1$, the compact-open topology on the space of continuous functions from $S^k$ to $T$ has a subbasis composed of sets of the form
\[
\bigl\{ f \st f\bigl(\closure{A}) \subseteq B \bigr\},
\]
where $A$ and $B$ are elements in the \emph{bases} $\cS^k$ and $\cT$, respectively. 
An effective enumeration of this subbasis can be constructed in a straightforward fashion from $s$ and $t$.

In particular, let $k \ge 1$ and let $s^k$ be an effective enumeration of $k$-tuples of basis elements derived from $s$.
Then a continuous function $f : (\Reals^k,\cS^k) \to (\Reals,\cT)$ is computable (under the compact-open topology) when
\[
\label{contfunct}
\bigl\{ \tuple{m,n} \st f\bigl(\closure{s^k(m)}) \subseteq t(n) \bigr\}
\]
is a c.e.\ set.
The set \eqref{contfunct} is the name of $f$.

A continuous function is computable in this sense if and only if it
is computable according to Definition~\ref{cfuncs1}.  (See
\citep[][Ch.~6]{358357} and \citep[][Thm.~3.2.14]{358357}).
Note that when $\cS = \cT = \RationalIntervals$, this recovers the standard definition of a computable real function.  When $\cS=\RationalIntervals$ and $\cT=\cR_\ro$, this recovers the standard definition of a lower-semicomputable real function \citep{MR1745071}.  

\subsection{Representations of Borel Probability Measures}

The following representations for probability measures on computable
topological spaces 
are devised from more general TTE representations in Schr\"oder
\citep{MR2351942} and Bosserhoff \citep{Bosserhoff08}, and agree with
 Weihrauch \citep{MR1694441} in the case of the unit interval. 
In particular, the representation for $\ProbMeasures(S)$ below is 
admissible with respect to the weak topology, hence
computably equivalent (see Weihrauch \citep[][Chap.~3]{358357})
to the canonical TTE representation for Borel measures given in
Schr\"oder \citep{MR2351942}. 

Schr\"{o}der~\citep{MR2351942}  has also shown the equivalence of this
representation for probability measures (as a computable space under the
weak topology) with \emph{probabilistic processes}.
A probabilistic process (see Schr\"oder and Simpson~\citep{1222755})
formalizes a notion of a program that uses randomness to
sample points in terms of their names of the form 
\eqref{comppoints}.

For a second-countable $T_0$ topological space $S$ with subbasis $\cS$,
let $\ProbMeasures(S)$ denote the set of Borel probability measures
on $S$
(i.e., the probability measures on the $\sigma$-algebra generated by
$\cS$).  
Such measures are determined by the measure they
assign to finite intersections of elements of $\cS$.  Note that $\ProbMeasures(S)$ is itself a second-countable $T_0$ space.

Now let $(S,\cS)$ be a computable topological space with respect to the enumeration $s$.   We will describe a
subbasis for $\ProbMeasures(S)$ that makes it a computable
topological space.
Let $\Alg_\cS$ denote the lattice generated by $\cS$ 
(i.e., the closure of $\cS$ under finite union and intersection), and let $s^\Alg$ be an effective enumeration derived from $s$.
Then, the class of sets
\[\textstyle \label{probsubbasis}
\{ \gamma \in \ProbMeasures(S) \st \gamma \sigma > q \},
\]
where  $\sigma \in \Alg_\cS$ and $q \in \Rationals$,
is a subbasis for the weak topology on $\ProbMeasures(S)$.
An effective enumeration of this subbasis can be constructed in a
straightforward fashion from the enumeration of $\cS$ and an effective enumeration $\{q_n\}_{n\in\omega}$ of the rationals,
making $\ProbMeasures(S)$ a computable topological space.  In particular, the name of a measure $\eta \in \ProbMeasures(S)$ is the set $\{ \tuple{m,n} \st \eta\bigl( s^\Alg(m) \bigr) > q_n \}$.

\begin{corollary}[Computable distribution]\label{cdist}
A Borel probability measure $\eta \in \ProbMeasures(S)$ is \emph{computable (under the weak topology)} if and only if $\eta B$ is a c.e.\ real, uniformly in the $s^\Alg$-index of $B \in \Alg_\cS$.
\qed\end{corollary}

Note that, for computable topological spaces $(S,\cS)$ and $(T,\cT)$ with enumerations $s$ and $t$, a measure $\eta \in \ProbMeasures(T)$ is computable relative to a point $x \in S$ when $\eta B$ is a c.e.\ real relative to $x$, uniformly in the $t^\Alg$-index of $B \in \Alg_\cT$.
Corollary~\ref{cdist} implies that the measure of a c.e.\ open set
(i.e., the c.e.\ union of basic open sets)
is a c.e.\ real (uniformly in the enumeration of the terms
in the union),
and that the measure of a co-c.e.\ closed set (i.e., the complement
of a c.e.\ open set) is a co-c.e.\ real (similarly uniformly); see,
e.g., \citep[][\S3.3]{945270} for details.
Note that on a discrete space, where singletons are both c.e.\ open
and co-c.e.\ closed, the measure of each singleton is a computable
real.  But for a general space, it is too strong to require that
even basic open sets have computable measure (see Weihrauch \citep{MR1694441} for a discussion;  moreover, such a requirement is stronger than what is necessary to ensure that  a, e.g., probabilistic Turing machine can produce exact samples to
arbitrary accuracy).

We will be interested in computable measures 
in $\ProbMeasures(S)$, where $S$ is either $\Reals^\omega$, $[0,1]^k$, or
$\ProbMeasures(\Reals)$. In order to 
apply Corollary~\ref{cdist}
to characterize concrete notions of
computability for $\ProbMeasures(S)$,
we will now describe choices of topologies on
these three spaces.

\subsubsection{Measures on Real Vectors and Sequences under the Standard Topology}

Using Corollary~\ref{cdist}, we can characterize the class of computable distributions 
on real sequences using the computable topological spaces
characterized above in Section~\ref{repreals}.
Let $\vec x=\{x_i\}_{i\ge1}$ be a sequence of real-valued random
variables (e.g., the exchangeable sequence $X$, or the derived random variables $\{V_\tau\}_{\tau
\in \RationalIntervals}$ under the canonical enumeration of
$\RationalIntervals$), and let $\eta$ be the joint distribution of $\vec x$.
Then $\eta$ is computable if and only if
$
\eta( \sigma \times \Reals^\omega) =
\Pr\bigl\{ x \in \sigma \times \Reals^\omega \bigr\}
$
is a c.e.\ real, uniformly in $k\ge 1$ and $\sigma \in \QAlgk{k}$.  
The following simpler characterization was given by M\"uller \citep[][Thm.~3.7]{310035}.
\begin{lemma}[Computable distribution under the standard topology]
\label{compdiststandard}\ \\
Let $\vec x=\{x_i\}_{i\ge1}$ be a sequence of real-valued random variables with joint distribution $\eta$.
Then $\eta$ is computable if and only if
\[\label{seqref2}
\eta( \tau \times \Reals^\omega)
=
\Pr\bigl({\textstyle \bigcap_{i=1}^k \{ x_i \in \tau(i) \}} \bigr)
\]
is a c.e.\ real, uniformly in $k\ge1$ and $\tau \in \RationalIntervals^k$.
\qed\end{lemma}
Therefore knowing the measure of the sets in $\bigcup_k
\RationalIntervals^k \subsetneq \bigcup_k \QAlgk{k}$ is sufficient.
Note that the right-hand side of \eqref{seqref2} is precisely the form of the 
left-hand side of the expression in Corollary~\ref{link1}. 
Note also that one obtains a characterization of the computability of a
finite-dimensional vector by embedding it as an initial segment of
a sequence.

\subsubsection{Measures on Real Vectors and Sequences under the Right Order Topology}

Borel measures on $\Reals$ under the right order topology play an important role when
representing measures on measures, as Corollary~\ref{cdist}
portends.

\begin{corollary}[Computable distribution under the right order topology] \label{rodist}
Let $\vec x=\{x_i\}_{i\ge 1}$ be a sequence of real-valued random variables with joint distribution $\eta$.  
Then $\eta$ is computable under the (product of the) right order topology if and only if 
\[\textstyle
\eta\bigl( \bigcup_{i=1}^m \bigl ( (c_{i1},\infty) \times \dotsm
\times (c_{ik},\infty) \times \Reals^\omega\bigr)\bigr) = \Pr \bigl( \bigcup_{i=1}^m 
 \bigcap_{j=1}^k \{x_j > c_{ij} \} \bigr)
\]
is a c.e.\ real, uniformly in 
$k,m \ge 1$ and $\vC = (c_{ij}) \in \Rationals^{m\times k}$.
\qed\end{corollary}

Again, one obtains a characterization of the computability of a
finite-dimen\-sional vector by embedding it as an initial segment of
a sequence.
Note also that if a distribution on
$\Reals^k$ is computable under the standard topology, then it is
clearly computable under the right order topology.
The above characterization is used in the next section as well as in
Proposition~\ref{mucomp}, where we must compute an integral with
respect to a topology that is coarser than the standard topology.

\subsubsection{Measures on Borel Measures}

The de~Finetti measure $\mu$ is the distribution of the directing random measure $\nu$, an 
$\ProbMeasures(\Reals)$-valued random variable.  Recall the definition $V_\RB \defas \nu \RB$, for $\RB \in \BorelSets$.
From Corollary~\ref{cdist}, it follows that $\mu$ is computable under the weak topology if and only if
\[\textstyle
\mu (\bigcup_{i=1}^m \bigcap_{j=1}^{k} \{ \gamma \in
\ProbMeasures(\Reals) \st \gamma \sigma(j) > c_{ij} \})
&= \Pr \bigl(\textstyle \bigcup_{i=1}^m \bigcap_{j=1}^{k} \{V_{\sigma(j)} > c_{ij}\} \bigr)
\label{boiled}
\]
is a c.e.\ real, uniformly in $k,m\ge 1$ and $\sigma \in \QAlg^k$ 
and
$\vC = (c_{ij}) \in \Rationals^{m\times k}$.
As an immediate consequence of \eqref{boiled} and
Corollary~\ref{rodist}, we obtain the following characterization of computable de~Finetti measures.
\begin{corollary}[Computable de~Finetti measure]\label{muiff}
The de~Finetti measure $\mu$ is computable relative to the joint
distribution of $\{V_\tau\}_{\tau\in\QAlg}$ under the right order
topology, and vice versa.
In particular,
$\mu$ is computable if and only if the joint distribution of $\{V_\tau\}_{\tau\in\QAlg}$ is computable under the right order topology.
\qed
\end{corollary}

\subsubsection{Integration}

The following lemma is a restatement of an integration result by
Schr\"oder \citep[][Prop.~3.6]{MR2351942}, which itself generalizes
integration results on standard topologies of finite-dimensional
Euclidean spaces by M\"uller \citep{310035} and the unit interval by
Weihrauch \citep{MR1694441}.

Define
\[
\bbI \defas \{ A \cap [0,1] \st A \in \RationalIntervals \},
\]
which is a basis for the standard topology on $[0,1]$,
and define
\[
\bbI_\ro \defas \{ A \cap [0,1] \st A \in \cR_\ro \},
\]
which is a basis for the right order topology on $[0,1]$.
\begin{lemma}[Integration of bounded lower-semicontinuous functions]\label{ceintegration}
Let $k\ge 1$ and let $\cS$ be either $\RationalIntervals$ or $\cR_\ro$.
Let 
\[
f : (\Reals^k,\cS^k) \to ([0,1], \bbI_\ro)
\]
be a continuous function and let $\mu$ be a 
Borel probability measure 
on $(\Reals^k, \cS^k)$.  Then 
\[
\int f \,d\mu
\]
is a c.e.\ real relative to $f$ and $\mu$.
\qed\end{lemma}
The following result of M\"uller~\citep{310035} is an immediate corollary.
\begin{corollary}[Integration of bounded continuous functions]\label{compintegration}
Let 
\[
g : (\Reals^k,\RationalIntervals^k) \to ([0,1],\bbI)
\]
be a continuous function and let $\mu$ be a 
Borel probability measure 
on $(\Reals^k, \RationalIntervals^k)$.
Then 
\[ 
\int g \,d\mu
\]
is a computable real relative to $g$ and $\mu$.
\qed\end{corollary}

\section{The Computable Moment Problem}
\label{momentproblem}

One often has access to the moments of a distribution, and wishes to
recover the underlying distribution.  
Let $\vec x = (x_i)_{i\in\omega}$ be a random vector in $[0,1]^\omega$ with
distribution $\eta$.  Classically, the distribution of $\vec x$ is
uniquely determined by the mixed moments of $\vec x$.  We show that
the distribution is in fact \emph{computable} from the mixed
moments. 

One classical way to pass from the moments of $\vec x$ to its
distribution is via the L\'evy inversion formula, which maps the
characteristic function $\phi_{\vec x} : \Reals^\omega \to
\mathbf{C}$, given by
\[
\phi_{\vec x}(t) \defas  \Expect(e^{i \langle t, \vec x \rangle}),
\]
to the distribution of $\vec x$.  However, even in the finite-dimensional
case, the inversion formula involves a limit
for which we have no direct handle on the rate of convergence, and
so the distribution it defines is not obviously computable.
Instead, we use a computable version of
the
Weierstrass approximation theorem to compute 
the distribution relative to
the mixed moments.

To show that $\eta$ is computable relative to the mixed moments,
it suffices to show that
$\eta(\sigma\times[0,1]^\omega) =  \Expect
\bigl(\Ind_{\sigma}(x_1,\dotsc,x_k)\bigr)$ is a c.e.\ real 
relative to the mixed moments, uniformly in $\sigma \in \bigcup_{k\ge1} \RationalIntervals^k$.
We begin by building sequences of polynomials that converge pointwise from below to indicator functions of the form $\Ind_\sigma$ for $\sigma \in \bigcup_{k\ge1} \QAlgk{k}$.

\begin{lemma}[Polynomial approximations] \label{effindicator}
Let 
$k\ge 1$ and 
$\sigma \in \QAlgk{k}$.
There is a sequence 
\[
\bigl\{ p_{n,\sigma} \st n\in \omega \bigr \}
\]
of rational polynomials of degree $k$,
computable uniformly in $n$, $k$, and $\sigma$, 
such that,
for all $\vec{x} \in [0,1] ^k$, we have 
\[
-2 \leq p_{n,\sigma}(\vec{x}) \leq \Ind_\sigma(\vec{x})
\qquad \text{and} \qquad
\lim_{m\to\infty}\, p_{m,\sigma}(\vec{x}) =\Ind_\sigma(\vec{x}).
\]
\end{lemma}

\begin{proof}
Let $k\ge 1$. 
For $\sigma \in \QAlgk{k}$, 
and $\vec x \in \Reals^k$, 
define $d(\vec x,[0,1]^k \setminus \sigma)$ to be the distance from $\vec x$ to the nearest point in $[0,1]^k \setminus \sigma$.  It is straightforward to show that
 $d(\vec x,[0,1]^k \setminus \sigma)$ is a computable real function of $\vec x$,
uniformly in $k$ and $\sigma$.

For $n\in\omega$, 
define $f_{n,\sigma} : \Reals^k \to \Reals$ by
\[
f_{n,\sigma}(\vec x) \defas -\frac 1 {n+1} + \min \{1, n\cdot d(\vec x,[0,1]^k\setminus \sigma) \},
\]
and note that $-1 \le f_{n,\sigma}(\vec x) \le \Ind_\sigma(\vec x) - \frac 1{n+1}$ and $\lim_{m\to \infty} f_{m,\sigma}(\vec x) = \Ind_\sigma(\vec x)$.
Furthermore, $f_{n,\sigma}(\vec x)$ is a computable (hence continuous) real
function of $\vec x$, uniformly in 
$n$, $k$, and 
$\sigma$. 

By the effective Weierstrass approximation theorem
(see Pour-El and Richards \citep[][p.~45]{MR1005942}), 
we can find (uniformly in $n$, $k$, and $\sigma$) a polynomial 
$p_{n,\sigma}$
with rational coefficients that uniformly approximates 
$f_{n,\sigma}$ to within $1/(n+1)$ on $[0,1]^k$.   These polynomials
have the desired properties.
\qed\end{proof}

We thank the anonymous referee for suggestions that simplified the proof of this lemma.

Using these polynomials, we can compute the distribution from the moments.  The other direction follows from computable integration results.

\begin{theorem}[Computable moments]
\label{computablemoments} 
Let $\vec x = (x_i)_{i\in\omega}$ be a random vector in
$[0,1]^\omega$ with distribution~$\eta$.
Then $\eta$ is 
computable relative to
the mixed moments of
$\{x_i\}_{i\in\omega}$, and vice versa.
In particular, $\eta$ is computable if and only if the mixed moments of
$\{x_i\}_{i\in\omega}$ are uniformly computable.
\end{theorem}
\begin{proof}
Any monic monomial in $k$ variables, considered as a real
function, computably maps $[0,1]^k$ into
$[0,1]$ (under the standard topology).  Furthermore, 
as the restriction of $\eta$ to any $k$ coordinates is computable
relative to $\eta$ (uniformly in the coordinates), it follows from
Corollary~\ref{compintegration} that each mixed moment (the
expectation of a monomial under such a restriction of $\eta$) is computable relative to $\eta$, uniformly in the index of the monomial and the coordinates.

Let $k\ge 1 $ and $\sigma \in \RationalIntervals^k$.
To establish the computability of $\eta$, it suffices to show that
\[
\eta (\sigma\times[0,1]^\omega) 
= \Expect \bigl (\Ind_{\sigma \times [0,1]^\omega} (\vec x) \bigr) 
= \Expect \bigl (\Ind_{\sigma} (x_1,\dotsc,x_k) \bigr).
\]
is a c.e.\ real relative to the mixed moments, uniformly in $k$ and $\sigma$.
By Lemma~\ref{effindicator}, there is a uniformly computable
sequence of polynomials $(p_{n,\sigma})_{n\in\omega}$ that converge
pointwise from below to the indicator $\Ind_{\sigma}$.  Therefore,
by the dominated convergence theorem,
\[\label{sup1234}
\Expect \bigl (\Ind_{\sigma} (x_1,\dotsc,x_k) \bigr)
= \sup_n \Expect \bigl ( p_{n,\sigma}(x_1,\dotsc,x_k) \bigr).
\]
The expectation $\Expect \bigl ( p_{n,\sigma}(x_1,\dotsc,x_k)
\bigr)$ is a $\Rationals$-linear combination of mixed moments, hence
a computable real relative to the mixed moments, uniformly in $n$, $k$, and $\sigma$. Thus the supremum \eqref{sup1234} is a c.e.\ real relative to the mixed moments, uniformly in $k$ and $\sigma$.
\qed\end{proof}


\section{Proof of the Computable de~Finetti Theorem}

For the remainder of the paper, let $X$ be a real-valued exchangeable
sequence with distribution $\chi$, let $\nu$ be its directing
random measure, and let $\mu$ be the corresponding de~Finetti measure. 

Classically, the joint distribution of $X$ is uniquely determined by
the de~Finetti measure (see Equation~\ref{mixture}).
We now show that the joint distribution of $X$ is in fact
\emph{computable} relative to the de~Finetti measure.

\begin{proposition}\label{mucomp} 
The distribution $\chi$ is 
computable relative to $\mu$.
\end{proposition}

\begin{proof}
Let $k\ge 1$ and $\sigma \in \RationalIntervals^k$. All claims are
uniform in $k$ and $\sigma$.
In order to show that $\chi$, the distribution of $X$, is
computable relative to $\mu$, we must show that 
$\Pr \bigl( \bigcap_{i=1}^k \{ X_i \in \sigma(i) \} \bigr )$
is a c.e.\ real relative to $\mu$.
Note that, 
by Corollary~\ref{link1},
\[\textstyle
\Pr \bigl( \bigcap_{i=1}^k \{ X_i \in \sigma(i) \} \bigr ) 
= \Expect \bigl ( \prod_{i=1}^k V_{\sigma(i)} \bigr ).
\]
Let $\eta$ be the joint distribution of $(V_{\sigma(i)})_{i\le k}$ and let $f : [0,1]^k \to [0,1]$ be defined by
\[\textstyle
f(x_1,\dotsc,x_k) \defas \prod_{i=1}^k x_i.
\]
To complete the proof, we now show that 
\[
\int \!f\, d\eta = \Expect \bigl ( {\textstyle \prod_{i=1}^k }V_{\sigma(i)} \bigr )
\]
is a c.e.\ real relative to $\mu$.
Note that $\eta$ is computable 
under the right order topology
relative to $\mu$.
Furthermore, $f$ is order-preserving (in each dimension) and lower-semicontinuous, i.e., is a continuous (and obviously computable) function from $([0,1]^k,\bbI_\ro^k)$ to $([0,1],\bbI_\ro)$.
Therefore, by Lemma~\ref{ceintegration}, we have that $\int f\,
d\eta$ is a c.e.\ real relative to $\mu$.
\qed\end{proof}

We will first prove the main theorem under the additional hypothesis
that the directing random measure is almost surely continuous.  We
then sketch a randomized argument that succeeds with probability
one.  Finally, we present the proof of the main result, which can be
seen as a derandomization.

\subsection{Almost Surely Continuous Directing Random Measures}
\label{easycase}

For $k\ge 1$ and $\Rpsi\in\RAlg^k$, we say that
$\Rpsi$ is a \emph{$\nu$-continuity set}
when, for $i\le k$, we have $\nu(\partial\Rpsi(i))=0$ a.s., where
$\partial\Rpsi(i)$ denotes the boundary of $\Rpsi(i)$.

\begin{lemma}\label{cemon1} 
Relative to $\chi$,
the mixed moments 
of $\{V_\tau\}_{\tau\in\QAlg}$ are uniformly c.e.\ reals 
and the mixed moments
of $\{V_{\closure{\tau}}\}_{\tau\in\QAlg}$ 
are uniformly co-c.e.\ reals; in particular, if  $\sigma \in
\QAlg^k$ (for $k\ge 1$)  is a $\nu$-continuity set, then the mixed moment
$\Expect \bigl( \prod_{i=1}^k V_{\sigma(i)}\bigr)$ is a computable
real, uniformly in $k$ and $\sigma$.
\end{lemma}
\begin{proof}
Let $k \ge 1$ and $\sigma \in \QAlg^k$. All claims are uniform in $k$ and $\sigma$.
By Corollary~\ref{link1}, 
\[\textstyle
\Expect \bigl (\prod_{i=1}^k  V_{\sigma(i)} \bigr ) = \Pr \bigl ( \bigcap_{i=1}^k  \{X_i \in \sigma(i)  \} \bigr ),
\]
which is a c.e.\ real relative to $\chi$.  
The set $\closure \sigma$ is a co-c.e.\ closed set in $\Reals^k$
because we can computably enumerate all $\tau\in\QAlg^k$ contained in the
complement of $\sigma$.
Therefore,
\[\textstyle
\Expect \bigl (\prod_{i=1}^k  V_{\closure{\sigma(i)}} \bigr ) 
 = \Pr \bigl ( \bigcap_{i=1}^k  \{X_i \in \closure{\sigma(i)}  \} \bigr )
\]
is the measure of a co-c.e.\ closed set, hence a co-c.e.\ real
relative to $\chi$.
When $\sigma$ is a $\nu$-continuity set,
\[\textstyle
\Expect \bigl( \prod_{i=1}^k V_{\sigma(i)}\bigr) = \Expect \bigl(
\prod_{i=1}^k V_{\closure{\sigma(i)}}\bigr),
\]
and so the expectation is a computable real relative to $\chi$.
\qed\end{proof}

\begin{proposition}[Almost surely continuous directing random measure]
\label{acCdeFtheorem}
Assume that $\nu$ is almost surely continuous.  
Then $\mu$ is 
computable relative to $\chi$.
\end{proposition}
\begin{proof}
Let $k\geq 1$ and $\sigma \in \QAlg^k$.
The almost sure continuity of $\nu$ implies that 
$\sigma$ is an $\nu$-continuity set.
Therefore, by Lemma~\ref{cemon1}, the moment
$\Expect \bigl ( \prod_{i=1}^k V_{\sigma(i)} \bigr)$  is a
computable real relative to $\chi$, uniformly in $k$ and $\sigma$.  The
computable moment theorem (Theorem~\ref{computablemoments}) then implies that
the joint distribution of the variables $\{V_{\tau}\}_{\tau \in \QAlg}$
is computable under the standard topology relative to $\chi$, and so
their joint distribution is also computable under the (coarser) right order topology relative to $\chi$. 
By Corollary~\ref{muiff}, this implies that $\mu$ is
computable relative to $\chi$.
\qed\end{proof}

\subsection{``Randomized'' Proof Sketch}

In general, the joint distribution of $\{V_\sigma\}_{\sigma\in\QAlg}$ is not computable under the standard topology because the directing random
measure $\nu$ may, with nonzero probability, have a point mass on a rational.
In this case, the mixed moments of $\{V_\tau\}_{\tau\in\QAlg}$ are
c.e., but not co-c.e., reals relative to~$\chi$.
As a result, the computable moment theorem
(Theorem~\ref{computablemoments}) is inapplicable. 
For arbitrary directing random measures, we give 
a proof of the computable de~Finetti theorem that works regardless
of the location of point masses.

Consider the following sketch of a ``randomized algorithm'':
We independently sample a countably infinite sequence of real numbers
$\boldA$ from a computable, absolutely continuous
distribution that has support everywhere on the real line (e.g., a
Gaussian or Cauchy).  Let $\Alg_\boldA$ denote the lattice generated
by open intervals with endpoints in $\boldA$.  Note that, with probability
one, $\boldA$ will be dense in $\Reals$ and every $\Rpsi \in
\Alg_\boldA$ will be a $\nu$-continuity set. 
If the algorithm proceeds analogously to the case where $\nu$ is almost
surely continuous, using $\Alg_\boldA$ as our basis, rather than
$\QAlg$, then it will compute the de~Finetti
measure with probability one.

Let $\bA$ be a dense sequence of reals such that $\nu(\bA) = 0$ a.s.
Consider the variables $V_\Rzeta$ defined in terms of elements
$\Rzeta$ of the new basis $\Alg_\bA$ (defined analogously to
$\Alg_\boldA$).
We begin by proving an extension of Lemma~\ref{cemon1}: 
The mixed moments of the set of variables $\{V_\zeta\}_{\zeta\in\Alg_\bA}$
 are computable
relative to $\bA$ and $\chi$.

\begin{lemma}\label{cemonreal} 
Let $k \ge 1$ and $\psi \in \Alg_\bA^k$.
The mixed moment 
$\Expect \bigl( \prod_{i=1}^k V_{\psi(i)}\bigr)$ 
is a computable real relative to $\bA$ and $\chi$, uniformly in
$k$ and $\psi$. 
\end{lemma}

\begin{proof}
Let $k\ge 1$ and $\psi \in \Alg^k_\bA$. All claims are uniform in
$k$ and $\psi$.
We first show that, relative to $\bA$ and $\chi$, the mixed moments
of $\{V_\zeta\}_{\zeta\in\Alg_\bA}$ are uniformly c.e.\ reals.
We can compute (relative to $\bA$) a sequence 
\[
\sigma_1, \sigma_2, \dotsc \in \QAlg^k
\]
 such that componentwise for each $n\ge 1$,
\[\textstyle
\sigma_n \subseteq \sigma_{n+1}
\qquad \text{and} \qquad
\bigcup_m \sigma_m = \psi.
\]
Note that if $\zeta, \varphi \in\QAlg$ satisfy
$\zeta \subseteq \varphi$,
then $V_{\zeta} \le V_{\varphi}$ (a.s.), 
and so, by the continuity of measures (and of multiplication), 
$\prod_{i=1}^k V_{\sigma_n(i)}$ converges from below to $\prod_{i=1}^k V_{\psi(i)}$
with probability one.
Therefore, the dominated convergence
theorem gives us
\[\textstyle
\label{supbA}
\Expect \bigl( \prod_{i=1}^k V_{\psi(i)}\bigr) 
= \sup_n \Expect \bigl( \prod_{i=1}^k V_{\sigma_n(i)}\bigr).
\]
Using Corollary~\ref{link1},
we see that the expectation $\Expect \bigl( \prod_{i=1}^k
V_{\sigma_n(i)}\bigr)$ is a
c.e.\ real relative to $\bA$ and $\chi$, 
uniformly in $n$,
and so the supremum 
\eqref{supbA} is a c.e.\ real relative to $\bA$ and $\chi$.

Similarly, the mixed moments of
$\{V_{\closure{\zeta}}\}_{\zeta\in\Alg_\bA}$ are uniformly co-c.e.\
reals relative to $A$ and $\chi$, 
as can be seen via a sequence of nested unions of
rational intervals whose union has complement equal to $\closure\psi$.  
Thus, because $\psi$ is a $\nu$-continuity set,
the mixed moment
$\Expect \bigl( \prod_{i=1}^k V_{\psi(i)}\bigr)$ 
is a computable real relative to $\bA$ and $\chi$. 
\qed\end{proof}

\begin{lemma}
\label{relchi}
The de~Finetti measure $\mu$ is 
computable relative to $\bA$ and $\chi$.
\end{lemma}
\begin{proof}
It follows immediately from Lemma~\ref{cemonreal} and Theorem~\ref{computablemoments} that the joint
distribution of $\{V_{\psi}\}_{\psi \in \Alg_\bA}$ is
computable relative to $\bA$ and $\chi$.
This joint distribution classically determines the de~Finetti
measure. 
Moreover, as we now show, we can compute (relative to $\bA$ and $\chi$) the desired
representation with respect to the (original) rational basis.
In particular, we prove that the joint distribution of $\{V_\tau\}_{\tau \in \QAlg}$ is computable under the right order topology relative to $\bA$ and $\chi$. 

Let $m,k \ge 1$, 
let $\tau \in \QAlg^k$, and
let $\vC = (c_{ij}) \in \Rationals^{m\times k}$.
We will express $\tau$ as a union of
elements of $\Alg^k_\bA$.
Note that $\tau$ is an c.e.\ open set (relative to $\bA$) with respect to the basis
$\Alg^k_\bA$. In particular, we can computably enumerate (relative to $\bA$,
and uniformly in $k$ and $\tau$) a sequence $\sigma_1, \sigma_2, \ldots \in \Alg^k_\bA$ such that $\cup_n \sigma_n = \tau$ and  $\sigma_n \subseteq \sigma_{n+1}$.  Note that $V_{\tau(j)} \ge V_{\sigma_n(j)}$ (a.s.)\ for all $n \ge 1$ and $j\le k$.
By the continuity of measures (and of union and intersection),
\[\textstyle
\label{supVtau}
\Pr \bigl ( \bigcup_{i=1}^m \bigcap_{j=1}^k \{V_{\tau(j)} > c_{ij}  \} \bigr )
= \sup_n \Pr \bigl ( \bigcup_{i=1}^m \bigcap_{j=1}^k \{V_{\sigma_n(j)} > c_{ij} \} \bigr ).
\]
The probability
$\Pr \bigl ( \bigcup_{i=1}^m \bigcap_{j=1}^k \{V_{\sigma_n(j)} >
c_{ij} \} \bigr )$ is a
c.e.\ real relative to $\bA$ and $\chi$, 
uniformly in $n$, $m$, $k$, $\tau$, and  $C$, and so the supremum 
\eqref{supVtau} is a c.e.\ real relative to $\bA$ and $\chi$,
uniformly in $m$, $k$, $\tau$, and $C$.
\qed\end{proof}

Let $\Phi$ denote the map taking $(\bA, \chi)$ to $\mu$, as
described in Lemma~\ref{relchi}.

Recall that $\boldA$ is a random dense sequence with a computable distribution, as defined above,
and let $\hat\mu=\Phi(\boldA,\chi)$.
Then $\hat \mu$ is a random variable, and moreover, $\hat \mu = \mu$
almost surely.   
However, while $\boldA$ is almost surely noncomputable,
the distribution of $\boldA$ is computable,
and so the distribution of $\hat \mu$ is computable relative to $\chi$.
Expectations with respect to the distribution of $\hat \mu$ can then be used to
(deterministically) compute $\mu$ relative to
$\chi$.

A proof along these lines could be made precise by making
$\ProbMeasures(\ProbMeasures(\ProbMeasures(\Reals)))$ into a computable
topological space.
Instead, in Section~\ref{mainproof}, we complete the proof by
explicitly computing $\mu$ relative to $\chi$ in terms of the
standard rational basis.  This construction can be seen as a
``derandomization'' of the above algorithm.

Alternatively, the above sketch could be interpreted as a degenerate
\emph{probabilistic process} (see Schr\"oder and Simpson
\citep{1222755}) that samples a name of the de~Finetti
measure with probability one.  Schr\"oder \citep{MR2351942} shows
that representations in terms of probabilistic processes are
computably reducible to representations of computable distributions.

The structure of the derandomized argument occurs in other proofs in computable analysis and probability theory. 
Weihrauch \citep[][Thm.~3.6]{MR1694441} proves a computable
integration result via an argument that could likewise be seen as a
derandomization of an algorithm that densely subdivides the unit
interval at random locations to find continuity sets.  
Bosserhoff \citep[][Lem.~2.15]{Bosserhoff08} uses a similar argument
to compute a basis for a computable metric space, for which every
basis element is a continuity set; this suggests an alternative
approach to completing our proof.
M\"uller \citep[][Thm.~3.7]{310035}
uses a similar construction to find open hypercubes 
such that for any $\epsilon > 0$, the probability on their
boundaries is less than $\epsilon$.
These arguments also resemble the
proof of the classical Portmanteau theorem
\citep[][Thm.~4.25]{MR1876169}, in which an uncountable family of sets
with disjoint boundaries is defined, almost all of which are continuity sets.

\subsection{``Derandomized'' Construction}
\label{mainproof}

Let $m,k \ge 1$ and $\vC = (c_{ij}) \in \Rationals^{m \times k}$.
By an abuse of notation, we define
\[\label{niceind}
\Ind_\vC : [0,1]^k \to [0,1]
\]
to be the indicator function for the set
\[\textstyle \label{niceset}
\bigcup_{i=1}^m (c_{i1},1] \times \dotsm \times (c_{ik},1].
\]
For $n \in \omega$, we denote by $p_{n,\vC}$ the polynomial $p_{n,\sigma}$ 
(as defined in Lemma~\ref{effindicator}),
where  
\[\textstyle
\sigma \defas \bigcup_{i=1}^m (c_{i1}, 2) \times \cdots \times (c_{ik},2) \in \QAlgk{k}.
\]
Here, we have arbitrarily chosen $2 > 1$ so that the sequence of polynomials $\{p_{n,C}\}_{n\in\omega}$ converges pointwise from below to $\Ind_\vC$ on $[0,1]^k$.

Let $\vec x=(x_1, \ldots, x_k)$ and $\vec y =(y_1,\dotsc,y_k)$.
We can write 
\[
p_{n,\vC}(\vec x)  = p^+_{n,\vC}(\vec x) - p^{-}_{n,\vC}(\vec x),
\]
where $p^+_{n,\vC}$ and $p^-_{n,\vC}$ are polynomials with positive coefficients. 
Define the $2k$-variable polynomial
\[
q_{n,\vC}(\vec x, \vec y) \defas p^+_{n,\vC}(\vec x) - p^{-}_{n,\vC}(\vec y).
\]
We denote
\[
q_{n,\vC}(\RealV_{\RZ(1)}, \ldots, \RealV_{\RZ(k)}, \RealV_{\Rzeta(1)}, \ldots, \RealV_{\Rzeta(k)})
\]
by
$q_{n,\vC}(\RealV_{\RZ},\RealV_{\Rzeta})$, and similarly with $p_{n,\vC}$.

\begin{proposition}
Let $n \in \omega$, let $k,m\ge 1$, let $\sigma\in\QAlg^k$,
and let $\vC \in\Rationals^{m \times k}$.
Then $\Expect q_{n,\vC}
(V_\sigma, V_{\closure{\sigma}})$
is 
a c.e.\ real relative to $\chi$,
uniformly in $n$, $k$, $m$, $\sigma$, and $\vC$.
\label{polypoly}
\end{proposition}
\begin{proof}
By Lemma~\ref{cemon1}, relative to $\chi$, 
and uniformly in $n$, $k$, $m$, $\sigma$, and $\vC$,
each monomial of
$p^+_{n,\vC}(V_{\sigma})$
has a c.e.\ real expectation, and each monomial of
$p^-_{n,\vC}(V_{\closure{\sigma}})$  has a co-c.e.\ real
expectation, and so by the linearity of expectation $\Expect q_{n,\vC}
(V_\sigma, V_{\closure{\sigma}})$
is a c.e.\ real.
\qed\end{proof}

In the final proof we use the following dense partial order on
products of $\RAlg$.

\begin{definition}
Let $k\ge 1$.
We call $\Rpsi \in \RAlg^k$ a \emph{refinement} of $\RZ \in
\RAlg^k$, and write $\Rpsi \arrowlt \RZ$, when
\[
\closure{\Rpsi(i)} \subseteq \RZ(i)
\]
for all $i\le k$. 
\end{definition}

We are now ready to prove the main theorem.

\begin{cdfproof}[Computable de~Finetti]
The distribution $\chi$
(of the exchangeable sequence $X$)
is computable relative to the de~Finetti measure
$\mu$ by Proposition~\ref{mucomp}.
We now give a proof of the other direction, 
showing that the joint distribution of
$\{V_\sigma\}_{\sigma\in\QAlg}$ is computable 
under the right order topology
relative to $\chi$, which by Corollary~\ref{muiff} will complete the proof.

Let $k,m\ge 1$, let $\pi \in \QAlg^k$, and let $\vC = (c_{ij}) \in \Rationals^{m\times k}$.  
For $\Rzeta \in \RAlg^k$, let $V_{\Rzeta}$ denote the $k$-tuple $(V_{\Rzeta(1)},\dotsc,V_{\Rzeta(k)})$ and similarly for $V_{\closure{\Rzeta}}$.  Take $\Ind_C$ to be defined as above in~\eqref{niceind} and \eqref{niceset}.
It suffices to show that
\[\textstyle \label{e0}
\Pr \left (\bigcup_{i=1}^m \bigcap_{j=1}^k \{V_{\pi(j)} > c_{ij} \} \right ) 
= \Expect \Ind_\vC(V_{\pi})
\]
is a c.e.\ real relative to $\chi$, uniformly in $k$, $m$, $\pi$, and $C$. 
We do this by a series of reductions, which results in a supremum
over quantities of the form
$\Expect q_{n,C} (V_\sigma, V_{\closure{\sigma}})$ for $\sigma\in\QAlg^k$.

By the density of the reals and the continuity of measures,
we have that
\[\label{nucont}
V_{\pi} = \sup_{\Rpsi \arrowlt \pi} V_{\Rpsi} \quad \mathrm{a.s.},
\]
where $\Rpsi$ ranges over $\RAlg^k$.
It follows that
\[
\Ind_\vC(V_{\pi}) = \sup_{\Rpsi \arrowlt \pi} \Ind_\vC(V_{\Rpsi}) \quad \mathrm{a.s.},
\]
because $\Ind_\vC$ is lower-semicontinuous and order-preserving (in
each dimension),
as \eqref{niceset} is an open set in the right order topology on $[0,1]^k$.
Therefore, by the dominated convergence theorem, we have that
\[\label{e1}
\Expect \Ind_\vC(V_{\pi})
= \sup_{\Rpsi\arrowlt\pi} \Expect \Ind_\vC(V_{\Rpsi}).
\]
Recall that the polynomials $\{p_{n,C}\}_{n \in \omega}$ converge pointwise from below to 
$\Ind_\vC$ in $[0,1]^k$.
Therefore, by the dominated convergence theorem,
\[\label{e2}
\Expect \Ind_\vC(V_{\Rpsi})
= \sup_n \Expect p_{n,C} (\RealV_\Rpsi).  
\]
As $\RealV_{\closure{\Rpsi(i)}} \ge
\RealV_{\Rpsi(i)}$  a.s.\ for $i\le k$, we have that
\[
\Expect p_{n,\vC} (\RealV_\Rpsi)
&=
\Expect p^+_{n,\vC} (\RealV_\Rpsi) - 
\Expect p^-_{n,\vC}
(\RealV_{\Rpsi}) \\
&\ge \label{newbound12}
\Expect p^+_{n,\vC} (\RealV_\Rpsi) - 
\Expect p^-_{n,\vC}
(\RealV_{\closure{\Rpsi}}).
\]
Note that if $\Rpsi$ is a $\nu$-continuity set, then 
$V_{\closure{\Rpsi(i)}} = V_{\Rpsi(i)}$ 
a.s., and so
\[
\Expect p_{n,\vC} (\RealV_\Rpsi) &= 
\Expect p^+_{n,\vC} (\RealV_\Rpsi) - 
\Expect p^-_{n,\vC}
(\RealV_{\closure{\Rpsi}}).
\]
Again, 
dominated convergence theorem 
gives us
\[ 
\Expect \left(
{\textstyle \prod_{i=1}^k V_{\Rpsi(i)} }
\right ) &=
 \sup_{\sigma\arrowlt\Rpsi}\Expect \left (
{\textstyle \prod_{i=1}^k V_{\sigma(i)} }
\right )
\quad \text{and}\\
 \Expect \left(
{\textstyle \prod_{i=1}^k V_{\closure{\Rpsi(i)}} }
\right )&=
 \inf_{\tau\arrowgt\Rpsi}
 \Expect \left (
{\textstyle \prod_{i=1}^k V_{\closure{\tau(i)}} }
\right ),
\]
where $\sigma$ and $\tau$ range over $\QAlg^k$.
Therefore, by the linearity of expectation,
\[
\Expect p^+_{n,\vC}(\RealV_\Rpsi) &=
 \sup_{\sigma\arrowlt\Rpsi}
\Expect p^+_{n,\vC}(V_{\sigma}) 
\quad \text{and}\\
\Expect p^-_{n,\vC}(\RealV_{\closure{\Rpsi}}) &=
 \inf_{\tau\arrowgt\Rpsi}
\Expect p^-_{n,\vC}(V_{\closure{\tau}}),
\]
and so, if $\Rpsi$ is a $\nu$-continuity set, we have that 
\[
\Expect p_{n,\vC} (\RealV_\Rpsi) 
&= \sup_{\sigma\arrowlt\Rpsi}
\Expect p^+_{n,\vC}(V_{\sigma})-
 \inf_{\tau\arrowgt\Rpsi}
\Expect p^-_{n,\vC}(V_{\closure{\tau}}) \\
 &= \label{ub1234}
\sup_{\sigma\arrowlt\Rpsi\arrowlt\tau} \Expect q_{n,\vC}
(V_\sigma, V_{\closure{\tau}}). 
\]
Because $\nu$ has at most countably many point masses, those
$\Rpsi\in\RealIntervals^k$ that are $\nu$-continuity sets
are dense in $\RationalIntervals^k$.  On the other hand, for those $\Rpsi$ that are not $\nu$-continuity sets, \eqref{ub1234} is a lower bound, as can be shown from \eqref{newbound12}.  Therefore,
\[\label{e3}
\sup_{\Rpsi\arrowlt\pi} 
\Expect p_{n,\vC} (\RealV_\Rpsi) =\sup_{\Rpsi\arrowlt\pi}\,
\sup_{\sigma\arrowlt\Rpsi\arrowlt\tau} \Expect q_{n,\vC}
(V_\sigma, V_{\closure{\tau}}).
\]
Note that 
$
\{(\sigma, \tau) \st (\exists\Rpsi\arrowlt\pi)\ 
\sigma\arrowlt\Rpsi\arrowlt\tau\} = 
\{(\sigma, \tau) \st \sigma\arrowlt\pi \mathrm{~and~} \sigma \arrowlt \tau \}.
$
Hence
\[\label{e4}
\sup_{\Rpsi\arrowlt\pi}\,  
\sup_{\sigma\arrowlt\Rpsi}\, \sup_{ \tau \arrowgt\Rpsi} \Expect q_{n,\vC}
(V_\sigma, V_{\closure{\tau}})
=
\sup_{\sigma\arrowlt\pi}\,\sup_{ \tau \arrowgt\sigma} \Expect q_{n,\vC}
(V_\sigma, V_{\closure{\tau}}).
\]
Again by dominated convergence we have
\[\label{e5}
\sup_{ \tau \arrowgt \sigma} \Expect q_{n,\vC}
(V_\sigma, V_{\closure{\tau}})
=\Expect q_{n,\vC}
(V_\sigma, V_{\closure{\sigma}}).
\]
Combining \eqref{e0}, \eqref{e1}, \eqref{e2}, \eqref{e3}, \eqref{e4}, and \eqref{e5},
we have 
\[
\label{finalsup}
\Expect \Ind_\vC(V_{\pi})
=
\sup_n\,
\sup_{\sigma\arrowlt\pi} \Expect q_{n,\vC}
(V_\sigma, V_{\closure{\sigma}}).
\]   
Finally, by Proposition~\ref{polypoly}, the expectation
\[
\Expect q_{n,\vC}
(V_\sigma, V_{\closure{\sigma}})
\]
is a c.e.\ real relative to $\chi$, uniformly in
$\sigma$, $n$, $k$, $m$, $\pi$, and $C$.
Hence the supremum \eqref{finalsup} is a c.e.\ real
relative to $\chi$, uniformly in  $k$, $m$, $\pi$, and $C$.
\qed\end{cdfproof}


\section{Exchangeability in Probabilistic Functional Programming Languages}
\label{funcpure}

The computable de~Finetti theorem has implications for the semantics
of probabilistic functional programming languages, and in particular, gives
conditions under which it is possible to eliminate modifications of
non-local state.
Furthermore, an implementation of the computable de~Finetti theorem
itself performs this code transformation automatically.

For context, we provide some background on 
probabilistic functional programming languages.  We then describe the code transformation
performed by the computable de~Finetti theorem, using the example
of the P\'olya urn and Beta-Bernoulli process discussed earlier.
Finally, we discuss partial exchangeability and its role in recent
machine learning applications. 

\subsection{Probabilistic Functional Programming Languages}
Functional programming languages with probabilistic choice operators have
recently been proposed as universal languages for statistical
modeling  (e.g., IBAL \citep{Pfeffer01}, $\lambda_\circ$\citep{1452048}, 
Church \citep{GooManRoyBonTen2008}, and HANSEI
\citep{DBLP:conf/dsl/KiselyovS09}).  
Within domain theory, researchers have considered idealized functional languages that can
manipulate exact real numbers, such as 
Escard\'o's {\sc RealPCF+} \citep{MR1650876}
(based on Plotkin \citep{MR484798}), and functional
languages have also been extended by
probabilistic choice operators (e.g., by 
Escard\'o \citep{citeulike:5467401} 
and Saheb-Djahromi \citep{MR519861}).

The semantics of probabilistic programs have been studied
extensively in theoretical computer science in the context of
randomized algorithms, probabilistic model checking, and other
areas.
However, the application of probabilistic programs to
universal statistical modeling has a somewhat different character
from much of the other work on probabilistic programming languages.

In Bayesian analysis, the goal is to use observed data to understand 
unobserved
variables in a probabilistic model.  This type of inductive reasoning, from
evidence to hypothesis, can be thought of as inferring the hidden states
of a program that generates the observed output.
One speaks of the 
\emph{conditional execution} of
probabilistic programs,  in which they are ``run backwards'' to
sample from the conditional probability distribution given the
observed data.

A wide variety of algorithms implement
conditional inference in probabilistic functional programming.
Goodman et~al.~\citep{GooManRoyBonTen2008} describe the language
Church, which extends a pure subset of Scheme, 
and whose implementation MIT-Church performs approximate conditional
execution via Markov chain Monte Carlo
(which can be thought of as a random walk over the execution of a
Lisp machine).
Park, Pfenning, and Thrun~\citep{1452048} describe the language
$\lambda_\circ$, which extends OCaml,
and they implement approximate conditional execution
by Monte Carlo importance sampling.
Ramsey and Pfeffer~\citep{Ramsey2002} describe a stochastic lambda calculus whose semantics are given by \emph{measure terms}, which support the efficient computation of conditional expectations.

Finally, in nonparametric Bayesian statistics, higher-order distributions
(e.g., distributions on distributions, or distributions on trees)
arise naturally, and so it is helpful to work in a language that can
express these types.  Probabilistic functional programming languages
are therefore a convenient choice for expressing nonparametric models.

The representation of distributions by randomized algorithms that
produce samples can highlight algorithmic issues.  For example,
a distribution will, in general, have many different representations
as a probabilistic program, each with its own time, space, and
entropy complexity.  For example, both ways of sampling a
Beta-Bernoulli process described in Section~\ref{polyaex} can be
represented in, e.g., the Church probabilistic programming language.
One of the questions that motivated the present work was whether 
there is always an algorithm for sampling from the de~Finetti
measure when there is an algorithm for sampling the exchangeable sequence.
This question was first
raised by Roy et al.~\citep{ICML}. 
The computable de~Finetti theorem answers this question in the
affirmative, and, furthermore, shows that one can move between these
representations automatically.  In the following section, we provide
a concrete example of the representational change made possible  by
the computable de~Finetti transformation, using the syntax of the
Church probabilistic programming language.

\subsection{Code Transformations}

Church extends a pure subset of Scheme (a dialect of Lisp) with a stochastic,
binary-valued\footnote{The original Church paper defined the \flip\ procedure to
return {\tt true} or {\tt false}, but it is easy to move between these two
definitions.} \flip\ procedure, calls to which return independent, Bernoulli$(\frac 12)$-distributed random values in $\{0,1\}$.
Using the semantics of Church, it is possible to associate every closed Church expression (i.e., one without free variables) with a distribution on values.
For example,  evaluations of the expression
\begin{quote}
{\tt (+~(\flip) (\flip) (\flip))} 
\end{quote}
produce samples from the Binomial$(n=3,\,p=\frac12)$ distribution, while
evaluations of
\begin{quote}
{\tt ($\lambda$ (x) (if (= 1 (\flip)) x 0))} 
\end{quote}
always return a procedure, applications of which behave like the probability kernel $x \mapsto \frac12(\delta_x + \delta_0)$, 
where $\delta_r$ denotes the Dirac measure concentrated on the real~$r$.  Church
is call-by-value and so evaluations of
\begin{quote}
{\tt (= (\flip) (\flip))} 
\end{quote}
return {\tt true} and {\tt false} with equal probability, while the application of  the procedure 
\begin{quote}
{\tt ($\lambda$ (x) (= x x))}
\end{quote}
to the argument {\tt (\flip)}, written 
\begin{quote}
{\tt (($\lambda$ (x) (= x x)) (\flip))},
\end{quote}
always returns {\tt true}. (For more examples, see \citep{GooManRoyBonTen2008}.)

In Scheme, unlike Church, 
one can modify the state of a non-local variable using
mutation via the {\tt set!}\ procedure.  (In functional
programming languages, non-local state may be implemented via other
methods.  For example, in Haskell, one could use the state monad.)
If we consider introducing a {\tt set!}\ operator to Church, thereby
allowing a procedure to modify its environment using mutation,
it is not clear how one can, in a manner similar to above, associate procedures with probability kernels and closed expressions with distributions.
For example, a procedure could then keep a counter variable and return an increasing
sequence of integers on repeated calls.  Such a procedure would not correspond
with a probability kernel.

A generic way to translate code with mutation into code without mutation
is to perform a state-passing transformation, where the state is
explicitly threaded throughout the program. In particular, a
variable representing state is passed into all procedures as
an additional argument, transformed in lieu of {\tt set!}\
operations, and returned alongside the original return values at the end of
procedures.  Under such a transformation, the procedure in the
counter variable example would be transformed into one that accepted the current count and returned the incremented count.
One downside of such a transformation is that it obscures
conditional independencies in the program, and thus complicates 
inference from an algorithmic standpoint.

An alternative transformation 
is made possible by 
the computable de~Finetti theorem, which implies that a particular type of
\emph{exchangeable} mutation can be removed \emph{without} requiring
a state-passing transformation.  Furthermore, this alternative
transformation exposes the conditional independencies.
The rest of this section describes a concrete example of this
alternative transformation, and builds on the mathematical characterization of the Beta-Bernoulli process and the P\'{o}lya urn scheme as described in Section~\ref{polyaex}.

Recall that the P\'{o}lya urn scheme induces the Beta-Bernoulli
process, which can also be described directly as a sequence of
independent Bernoulli random variables with a shared parameter
sampled from a Beta distribution.  In Church it is possible to write
code corresponding to both descriptions, but expressing the P\'olya
urn scheme without the use of mutation requires that we keep track
of the counts and thread these values throughout the sequence.  If
instead we introduce the {\tt set!}\ operator and track the number
of red and black balls by mutating non-local state, we can compactly
represent the P\'{o}lya urn scheme in a way that mirrors the form of
the more direct description using Beta and Bernoulli random variables.

Fix $a, b> 0$, and define {\tt sample-beta-coin} and {\tt
sample-p\'olya-coin} as follows:

\noindent\newline
\begin{tabular}{cc}
\begin{minipage}{5.7cm}
\emph{(i)}
\begin{Verbatim}[commandchars=\\\{\},codes={\catcode`$=3\catcode`^=7}] 
({\bf\ttfamily{define}} (sample-beta-coin)
  ({\bf\ttfamily{let}} ((weight ({\sl\ttfamily{beta}} a b)))
    ($\lambda$ () ({\sl\ttfamily{flip}} weight)) ) )





\end{Verbatim}
\end{minipage}
&
\begin{minipage}{8.0cm}
\emph{(ii)}
\begin{Verbatim}[commandchars=\\\{\},codes={\catcode`$=3\catcode`^=7}] 
(define (sample-p\'olya-coin)
  (let ((red a)
        (total (+ a b)) )
    ($\lambda$ () (let ((x ({\sl\ttfamily{flip}} $\frac{\mathtt{red}}{\mathtt{total}}$)))
            (set! red (+ red x))
            (set! total (+ total 1))
            x ) ) )

\end{Verbatim}
\end{minipage}
\end{tabular}

\noindent Recall that, given a Church expression $E$, the evaluation of the {\tt($\lambda$ () $E$
)} special form in an environment $\rho$ creates a procedure of no arguments whose
application results in the evaluation of the expression $E$ in the
environment $\rho$.
The application of either {\tt sample-beta-coin} or
{\tt sample-p\'olya-coin} returns a procedure of no arguments 
whose application returns (random) binary values.  In particular, if 
we sample two procedures {\tt my-beta-coin} and {\tt my-p\'olya-coin} via

\begin{quote}
{\tt (define my-beta-coin (sample-beta-coin))}\\
{\tt(define my-p\'olya-coin (sample-p\'olya-coin))}
\end{quote}
then repeated applications of both 
{\tt my-beta-coin} and {\tt my-p\'olya-coin} produce random binary sequences that are Beta-Bernoulli processes.

Evaluating {\tt (my-beta-coin)} returns 1 with
probability {\tt weight} and 0 otherwise, where the shared
{\tt weight} parameter is itself drawn from a Beta$(a,b)$
distribution on $[0,1]$.  
The sequence induced by repeated applications of {\tt my-beta-coin} is exchangeable because applications of {\sl\ttfamily flip} return
independent samples. Note that the sequence is \emph{not} i.i.d.;  for example, 
an initial sequence of ten 1's would lead one to predict that the next
application is more likely to return 1 than 0.  However, conditioned on 
{\tt weight} (a variable hidden within the opaque procedure
{\tt my-beta-coin}) the sequence is i.i.d.  
If we sample another procedure, {\tt my-other-beta-coin}, via 
\begin{quote}
{\tt(define my-other-beta-coin (sample-beta-coin))}
\end{quote}
then its corresponding {\tt weight} variable will be independent,
and so repeated applications will generate a sequence that is independent of that generated by {\tt my-beta-coin}.

The code in \emph{(ii)} implements the P\'{o}lya urn scheme with $a$
red balls and $b$ black balls (see
\citep[][Chap.~11.4]{MR1093667}), and so the sequence of return values from repeated applications of {\tt my-p\'olya-coin} is exchangeable.  Therefore,
de~Finetti's theorem implies that 
the distribution of the sequence is equivalent to that induced by
i.i.d.\ draws from the directing random measure. 
In the case of the P\'{o}lya urn scheme, we know that the directing random
measure is a random Bernoulli measure whose parameter has a
Beta$(a,b)$ distribution.
In fact, the (random) distribution of each sample produced by 
{\tt my-beta-coin} is such a random Bernoulli measure.
Informally, we can therefore think of
{\tt sample-beta-coin}
as the de~Finetti measure of the
Beta-Bernoulli process.

Although the distributions on sequences induced by {\tt my-beta-coin} and
\linebreak
{\tt my-p\'olya-coin} are identical, there is an important semantic difference between these two
implementations caused by the use of {\tt set!}.  While applications
of {\tt sample-beta-coin} produce samples from the de~Finetti
measure in the sense described above,
applications of {\tt sample-p\'olya-coin} do not; successive applications of {\tt
my-p\'olya-coin} produce samples from different distributions, none of which is 
the
directing random measure for the sequence (a.s.).
In particular, the distribution on return values changes  
each iteration
as the sufficient statistics are updated 
(using the mutation operator {\tt set!}). 
In contrast, applications of {\tt my-beta-coin} do not modify non-local state;
in particular, the sequence produced by such applications is i.i.d.\ conditioned on the variable {\tt weight}, which does not change during the course of execution. 

An implementation of the computable de~Finetti theorem
(Theorem~\ref{CdeFtheorem}),
specialized to the case of binary sequences (in which case the de~Finetti measure is a distribution on Bernoulli measures and is thus determined by the distribution on $[0,1]$ of the random probability assigned to the value 1), transforms \emph{(ii)}
into a mutation-free procedure whose return values have the same
distribution as that of the samples produced by evaluating {\ttfamily ({\sl\ttfamily beta} a b)}.  

In the general case, given a program that generates an exchangeable sequence of reals, an implementation of the computable de~Finetti theorem produces
a mutation-free procedure {\tt generated-code} such that applications of the procedure {\tt sample-directing-random-measure} defined by

\begin{quote}
\begin{Verbatim}[commandchars=\\\{\},codes={\catcode`$=3\catcode`^=7}] 
({\bf\ttfamily{define}} (sample-directing-random-measure)
  ({\bf\ttfamily{let}} ((shared-randomness ({\sl\ttfamily{uniform}} 0 1)))
    ($\lambda$ () (generated-code shared-randomness)) ) )
\end{Verbatim}
\end{quote}
sample from the de~Finetti measure in the sense described above.
In particular, \emph{(ii)} would be transformed into a procedure  {\tt
generated-code} such that the sequences produced by repeated applications of the
procedures returned by 
{\tt sample-beta-coin} 
and 
{\tt sample-directing-random-measure} 
have the same distribution.

In addition to their simpler semantics, mutation-free procedures 
are often desirable for practical reasons. For example, having                 
sampled the directing random measure,                                           
an exchangeable sequence of random variables can be efficiently 
sampled in parallel without the overhead necessary to communicate               
sufficient statistics.  
Mansinghka~\citep{vkmthesis} describes some situations
where one can exploit conditional independence and exchangeability
in probabilistic programming languages for improved parallel
execution.

\subsection{Partial Exchangeability of Arrays and Other Data Structures}
\label{partialexch}

The example above involved binary sequences, but the computable de
Finetti theorem can be used to transform implementations of real
exchangeable sequences.  
Consider the following exchangeable sequence whose combinatorial structure is known as the Chinese restaurant process 
(see Aldous \citep{MR883646}).
Let $\alpha > 0$ be a computable real and let $H$ be a computable
distribution on $\Reals$.  
For $n\ge 1$, each $X_n$ is sampled in turn according to the conditional
distribution
\[
\Pr[X_{n+1} \given X_1,\dotsc, X_n ] =
\frac{1}{n+\alpha} \Bigl( \alpha\,H + \sum_{i=1}^n \delta_{X_i} \Bigr)
 \qquad \text{a.s.}
\]
The sequence $\{X_n\}_{n\ge1}$ is exchangeable and the directing random measure is a Dirichlet process whose ``base measure'' is $\alpha H$.
Given such a program, we can automatically recover the underlying
Dirichlet process prior, samples from which are random measures
whose discrete structure was characterized by Sethuraman's
``stick-breaking construction'' \citep{Sethuraman1994}.   Note that the
random measure is not produced in the same manner as Sethuraman's construction
and certainly is not of closed form.  But the resulting mathematical
objects have the same structure and distribution. 

Exchangeable sequences of random objects other than reals can often
be given de~Finetti-type representations.
For example, the Indian buffet process, defined by Griffiths and
Ghahramani \citep{Griffiths05infinitelatent}, is the combinatorial
process underlying a \emph{set-valued} exchangeable sequence that
can be written in a way analogous to the P\'{o}lya urn in
\emph{(ii)}.
Just as the Chinese restaurant process gives rise to the Dirichlet process,
the Indian buffet process gives rise to the Beta process (see Thibaux and Jordan
\citep{Thibaux2007} for more details).

In the case where the ``base measure'' of the underlying Beta
process is \emph{discrete}, the resulting exchangeable sequence of
sets corresponds to an exchangeable sequence of \emph{integer}
indices (encoding finite subsets of the countable support of the discrete base measure).  If we are given such a representation, the
computable de~Finetti theorem implies the existence of a computable
de~Finetti measure.

However, the case of a general base measure is more complicated.
A ``stick-breaking construction'' of the Indian buffet process given by 
Teh, G{\"{o}}r{\"{u}}r, and Ghahramani \citep{TehGorGha2007}  
is analogous to the code in
\emph{(i)}, but samples only a $\mathrm{\Delta}_1$-index for the
(a.s.\ finite) sets, rather than a canonical index (see Soare
\citep[][II.2]{MR882921}); however, many applications depend on
having a canonical index.
These observation were first noted by Roy et~al.~\citep{ICML}.  
Similar problems arise when using the Inverse L\'evy Measure method
\citep{MR1630084} to construct the Indian buffet process.
The computable de~Finetti theorem is not directly applicable 
in this case
because the theorem pertains only to exchangeable sequences of
real random variables, not random sets, although
an extension of the 
theorem to computable Polish spaces might suffice.

Combinatorial structures other than sequences have been given
de~Finetti-type representational theorems based on notions of 
\emph{partial} exchangeability.
For example, an array of random variables is called
\emph{separately} (or \emph{jointly}) exchangeable when its
distribution is invariant under (simultaneous) permutations of the
rows and columns and their higher-dimensional analogues.
Nearly fifty years after de~Finetti's result,
Aldous \citep{MR637937} and Hoover \citep{Hoover79}
showed that the entries of an infinite
array satisfying either separate or joint exchangeability
are conditionally i.i.d.
These results have been connected with the theory of graph limits by
Diaconis and Janson~\citep{DiaconisJanson} and 
Austin~\citep{MR2426176} 
by considering the adjacency matrix of an exchangeable random graph.

As we have seen with the Beta-Bernoulli process and other examples,
structured probabilistic models can often be represented in
multiple ways, each with its own advantages (e.g.,
representational simplicity,
compositionality, 
inherent parallelism, 
etc.).
Extensions of the computable de~Finetti theorem to partially
exchangeable settings could provide 
analogous transformations between representations
on a wider range of data structures, including many that are
increasingly used in practice.
For example, the Infinite Relational Model \citep{Kemp2006} can be
viewed as an urn scheme for a partially exchangeable array,
 while the hierarchical
stochastic block model constructed from a Mondrian process in \citep{RoyTeh2009} is described in a way that mirrors the Aldous-Hoover
representation, making the conditional independence explicit.


\subsubsection*{Acknowledgements} \ \\
C.E.F.\ has been partially supported by NSF Grant No.\ DMS-0901020, and
D.M.R.\ has been partially supported 
by an NSF Graduate Research Fellowship.
Some of the results in this paper were presented at the 
\emph{Computability in Europe}
conference in Heidelberg, Germany, July 19--24, 2009,
and an extended abstract 
\citep{DBLP:conf/cie/FreerR09} was published in the proceedings.
The authors would like to 
thank Nate Ackerman, Oleg Kiselyov, Vikash Mansinghka, Hartley Rogers, Chung-chieh
Shan,
and the anonymous referees of both the extended abstract and the present article for helpful comments.




\end{document}